\documentclass[12pt]{amsart}
\usepackage{textcase}
\usepackage{latexsym, cite}
\usepackage[T2A]{fontenc}
\usepackage[utf8]{inputenc}
\usepackage{amsmath,amsthm}
\usepackage{amsfonts, amssymb}
\usepackage{indentfirst}
\usepackage[english]{babel}

\def\bes{\begin{eqnarray*}}
\def\ees{\end{eqnarray*}}
\def\bee{\begin{eqnarray}}
\def\eee{\end{eqnarray}}
\def\la{\langle}
\def\ra{\rangle}

\def\O{\mathbf O}

\def\0{\bar 0}
\def\1{\bar 1}

\def\ctd{\hfill$\Box$}

\def\M{\mathfrak M}

\newtheorem{thm}{Theorem}[section]
\newtheorem{cor}[thm]{Corollary}
\newtheorem{lem}[thm]{Lemma}
\newtheorem{prop}[thm]{Proposition}

\theoremstyle{definition}

\theoremstyle{remark}

\numberwithin{equation}{section}
\begin{document}

\title[Alternative Veronese subalgebras]{The Veronese subalgebras of a free alternative algebra of finite rank are finitely generated}%
\author{S.\,V.\,Pchelintsev,\ I.\,P.\,Shestakov}%
\address{\textup{\scriptsize
S.\,V.\,Pchelintsev
\newline\indent The Financial University by the Goverment of the Russian Federation, 
\newline\indent Moscow, Russia}}
\email{pchelinzev@mail.ru}
\thanks{The first author supported by FAPESP, Proc. 2023/01159-5. The second author was supported by FAPESP, Proc. 2018/23690-6 and CNPq, Proc. 304313/2019-0, Brazil, and also by IMC SUSTech, Shenzhen, China }%
\address{\textup{\scriptsize
 I.\,P.\,Shestakov
\newline\indent Instituto de Matem\'atica e Estat\'{i}stica,
\newline\indent Universidade de S\~ao Paulo,
\newline\indent S\~ao Paulo, Brazil
\newline\indent {\tiny and}
\newline\indent Sobolev Institute of Mathematics,
\newline\indent Novosibirsk, Russia,
\newline\indent {\tiny and}
\newline\indent International Mathematical Center of SUSTech, 
\newline\indent Shenzhen, China}}
\email{shestak@ime.usp.br}
\begin{abstract}
It is proved that for any natural number $n$ the subalgebra of a free finitely generated alternative algebra generated by all the words on generators whose length is a multiple of $n$  {\em (the Veronese $n$-subalgebra)}, is finitely generated.
\end{abstract}
\maketitle


\begin{flushright}
Dedicated to Vladislav Kirillovich Kharchenko\\ on occasion of his 70-th annyversary
\end{flushright}

\section*{Introduction}

\hspace{\parindent}
In \cite{Khar}, V.K.Kharchenko investigated the properties of the subalgebras of invariants of a finitely generated free associative algebra under the action of homogeneous finite groups of automorphisms.

It occurs that these subalgebras are always free and very rarely are finitely generated. Namely: the subalgebra of invariants is finitely generated if and only if  the group is scalar. In other words, the only finitely generated invariant subalgebras are {\em  Veronese subalgebras}, that is, the subalgebras generated for a fixed number $n$ by the words on generators whose length is a multiple of $n$.

\smallskip

In this paper we consider the similar questions for the variety of alternative algebras.
\smallskip

 Let $F(\mathfrak{M})$ be a free algebra in a homogeneous variety of algebras $\mathfrak{M}$ on a set of free generators  $X=\{x_1,\ldots,x_r\ldots\}$.
Let $n\geq 2$ be a natural number. The subspace spanned by all the monomials whose length is a multiple of  $n$ is a subalgebra of  $F(\mathfrak{M})$ which is called  {\em the  Veronese $n$-subalgebra}. Denote it as  $V^{(n)}(\mathfrak{M})$.
 
It is clear that if  $\mathfrak{M}= \rm Ass$  is the variety of associative algebras then the subalgebra $V^{(n)}(\rm Ass)$ is generated by monomials of length $n$, hence in the case of finite number of generators it is finitely generated.

In nonassociative case the finite generation of Veronese subalgebras in algebras of finite rank is not evident at all. Moreover, it is not true, for example, for right alternative algebras or for Lie algebras.

\smallskip

The main result of this work is the following
\begin{thm}
Let $\rm Alt$ be the variety of alternative algebras.
Then the Veronese subalgebra $V^{(n)}(\rm Alt)$ in a free algebra of finite rank is finitely generated.
\end{thm}

\smallskip

The main part of the paper is devoted to the proof of this theorem.  In conclusion, we show that the invariant subalgebras of free algebras in the  varieties  $\M$ with basic rank $r_b(\M) > 2$ may be non-free. In particular, they are not free in alternative, Jordan, Malcev, and  (-1,1) algebras.


\hspace{\parindent}

\section{Reduction to ${PI}$-algebras}

\hspace{\parindent}
Below in this section  $A$ always denotes the free alternative algebra on a countable set of generators $X$.
\smallskip

Following \cite{ZSSS}, denote by $N=N(A)$ and $U=U(A)$ {\em the associative center} and {\em the associative nucleus} of the algebra $A$, that is,  the largest ideal of the algebra $A$ lying in $N(A)$, and by $D=D(A)$  the associator ideal of the algebra  $A$, that is, the ideal generated by the all associators $(a,b,c)=(ab)c-a(bc);\, a,b,c\in A$.  It is known (see \cite{ZSSS}) that
\bes
[N,A]\subseteq N,\  [N,N]\subseteq U,\  UD=DU=0.
\ees
Denote also by $V^{(n)}=V^{(n)}(A)$ the Veronese $n$-subalgebra of the algebra $A$ and set $N^{(n)}=N\cap V^{(n)}$.

\begin{lem}\label{lem1}
For any $t\geq (n-1)^2+1$, the inclusion   $N^t\subseteq N^{(n)}A+U$ holds.
\end{lem}
\begin{proof}
It suffices to prove that for any homogeneous elements $k_1,\ldots, k_t\in N$ the product $k_1\cdots k_t\in N^{(n)}A+U$. Notice that due to the inclusion $[N,N]\subseteq U$, the order of  factors in the product $k_1\cdots k_t$ does not matter. Let $d_i=\deg(k_i)$ and $r_i=rest(d_i,n)$ be the rest of the division of $d_i$ by $n$. If $r_{i_0}=0$ for some $i_0$ then  $k_{i_0}\in N^{(n)}$, and putting $k_{i_0}$ at the beginning of the product, we get the needed inclusion. Therefore, we can assume that  $r_i\in \{1,\ldots,n-1\}$ for all $i=1,\ldots t$, and hence
\bes
\{k_1,\ldots,k_t\}=K_1\cup K_2\cup\cdots\cup K_{n-1},
\ees
where $K_i=\{k_j\,|\,r_j=i,\  j=1,\ldots t\},\,  i=1,\ldots, n-1$.

By the definition of the number $t$, there exists a set $K_{i_0}$ which contains at least  $n$ elements. It is clear that the product of these   $n$ elements lies in $N^{(n)}$. Putting them at the beginning of the product, we get the needed statement.

\end{proof}

\begin{lem}\label{lem2}
In the notation of lemma \ref{lem1}
\bes
id\la N^{2t}\ra\subseteq V^{(n)}A+AV^{(n)},
\ees
where $id\la M\ra$ denotes the ideal of the algebra $A$, generated by the set $M$.
\end{lem}
\begin{proof}
Denote first that $N^tA^n\subseteq  V^{(n)}A+AV^{(n)}$.
Let $w\in N^{t}$ and let $v$ be a monomial from $A$ of degree $\deg v\geq n$. By lemma \ref{lem1} we have
\[
wv = \left( \sum  n_iv_i+u\right)v=\sum  (n_iv_i)v+uv,
\]
where $ n_i\in N^{(n)},\, u\in U,\,  v_i\in A$.

We prove that every summand on the right side of the last equation lies in  $ V^{(n)}A+AV^{(n)}$.
First, due to  equality $u\cdot D(A)=0$, 
the arrangement of parentheses in the monomial $v$ does not matter for the product $uv$.
Therefore, we can assume that $v=v'w$ where  $w$ is a monomial of length $n$. Then
\[
uv=u(v'w)=(uv')w\in AV^{(n)}.
\]
Secondly, since $n_i\in N^{(n)}$, we have
\[(n_iv_i)v = n_i(v_iv)\in V^{(n)}A.
\]

Thus we have proved that for any monomial  $v$ of length $d(v)\geq n$, the inclusion $wv\in V^{(n)}A+AV^{(n)}$ is true.  Now, let  $w_1,w_2\in N^t;\, a,b\in A$. Notice that $[N^t,A]\subseteq  N^t$. Arguing modulo $ V^{(n)}A+AV^{(n)}$, we get
\bes
a(w_1w_2)b &=& (w_1w_2)ab + [a,w_1w_2]b=w_1(w_2ab) + [a,w_1w_2]b\\
& \equiv& w_1[a,w_2]b+[a,w_1]w_2b \equiv [a,w_1] w_2 b \in N^tA^n\equiv 0
\ees
since $\deg w_i>n$.

\end{proof}

\begin{cor}\label{cor1}
In the notation of lemma \ref{lem1},
\bes
id_T\la [x,y]^{8t}\ra\subseteq V^{(n)}A+AV^{(n)},
\ees
where $id_T\la M\ra$ denotes the $T$-ideal of the algebra $A$ generated by the set $M$.
\end{cor}
\begin{proof}
It suffices to note that $N^t$ is a $T$-subalgebra of the algebra $A$ and $[x,y]^4\in N$ (see \cite{ZSSS}).
\end{proof}

\section{Finite generation of associator bimodules}
\hspace{\parindent}
Let $A$ be an alternative algebra and $M$ be an alternative $A$-bimodule.
Define by induction a descending chain of  $A$-subbimodules:
$M_0=M$, and for $k>0$  $M_{k}$ is the subbimodule of $M$ generated by the set $(M_{k-1},A,A)$.

Recall that in an alternative algebra the following aidentities hold (see \cite{ZSSS}):
\begin{equation}\label{eq_muf}
(rx, s, x) = x(r, s, x),\qquad (xr, s, x) =(r, s, x)x,
\end{equation}
\begin{equation}\label{eq_jord}
(r, s, x^2) = (r, s,x)\circ x ,
\end{equation}
\begin{equation}\label{eq_qv}
(r, s, xyx) = x(r, s,y)x + \{xy(r, s,x)\},
\end{equation}
where $\{abc\} = (ab) + (cb)a$. The identities \eqref{eq_muf}-\eqref{eq_qv} are called \emph{the Moufang identities}.

\begin{lem}\label{lem3}
If an algebra  $A$ is finitely generated and   $M$ is a finitely generated  $A$-bimodule,
then for any $k$ the bimodule $M_{k}$ is finitely generated as well.
\end{lem}
\begin{proof}
Without loss of generality, we may assume that $A$ is a free algebra on a finite set of generators $X$.
Let  $M$ be generated by the elements
\begin{center}
$m_1,\ldots, m_s$.
\end{center}
It suffices to prove that the bimodule $M_{1}$ is finitely generated.
Denote by  $T_M(A)$  a unital subalgebra of the algebra of linear transformations  $End\,(M)$ generated by the operators of right ad left multiplication:
\[R_a, L_a: M\rightarrow M,\quad mR_a=ma,\quad mL_a=am.\]
By the assumption, every element in  $M$ is a linear combination of the elements of the form
\begin{center}
$m_1\varphi_1,\ldots, m_s\varphi_s,$\quad where $\varphi_i\in T_M(A)$.
\end{center}
Let us prove that  $M_{1}$ is generated by the elements of the form
\begin{equation}\label{eq_2}
(m_iL_{u_1}R_{u_2}, u_3, u_4),
\end{equation}
where $u_j,\, j=\overline{1,4}$ are the regular $r_1$-words  on generators of $A$ (see \cite[5.5]{ZSSS});
 the words $u_1,u_2$ may be absent.
The elements of type \eqref{eq_2} we call {\em regular}.
The elements of the subbimodule of $M$ generated by regular elements we call {\em almost regular}.
Let us prove that an element of type  $m=(m_i\varphi, a, b)$, where  $\varphi$ is an operator word from $T_M(A)$ and  $a, b$ are monomials over  $X$,
can be represented as a sum of almost regular elements.

Define the degree of the element  ${\deg}_X(m)$ as the sum of degrees of monomials  $a, b$ and of monomials involved in the operator word  $\varphi$.
We will use the induction on the number $d={\deg}_X(m)$. The base of induction for  $d=2$ is evident.
Assume that the claim is proved for all numbers less than $d$ and consider an element $m$ of degree $d$. By the induction assumption and linearization of identities  \eqref{eq_jord} and \eqref{eq_qv}, and by 
 \cite[proposition 3]{ZSSS}, we can assume that the monomials  $a$ and $b$ are $r_1$-words.
Besides, by the induction assumption we may assume that  $\varphi = L_uR_v$, where the order of operators  $L_u$ and $R_v$ does not matter. Using the linearized Moufang identities, transform the element $(um_iv,a,b)$, where  $u,v$  are arbitrary monomials and  $a,b$ are $r_1$-words:
\[(u(m_iv),a,b) =  -(a(m_iv),u,b) + (m_iv,a,b)u + (m_iv,u,b)a, \]
\[((am_i)v,u,b) = - ((am_i)b,u,v) + v(am_i, u,b) + b(am_i,u,v).\]
Now the proof can be easily finished for any of the associators on the right side.

The lemma is proved.
\end{proof}

\section{The proof of the theorem}
\hspace{\parindent}

Let $A$ be  a free  alternative algebra on a finite set of generators $X$. It suffices to prove that for any $n\geq 1$ there exists a number $N_0$ such that
 \bes
A^{N_0}\subseteq V^{(n)}A+AV^{n)}.
\ees
In fact, in this case the subalgebra $V^{(n)}(A)$ is generated by the  (finite) set of monomials from  $V^{(n)}$ of length $\leq N_0$.

By corollary \ref{cor1} we can assume that the algebra $A$ satisfies the identity  $[x,y]^{m}=0$. Then due to \cite{She1981}, $A$ is associator nilpotent, that is, there exists $k\geq 0$ such that $D_{k}(A)=0$, where $D_{0}(A)=A$ and for $i>0$ by induction $D_{i}(A)$ is defined as the ideal of  $A$, generated by the set $(D_{i-1}(A),A,A)$.

It suffices to prove that for any $i\geq 0$ there exists $N_i$ such that $A^{N_i}\cap D_{i-1}(A)\subseteq V^{(n)}A+AV^{n)}+D_i(A)$.
By lemma \ref{lem3}, the ideal $D_{i-1}(A)$ is finitely generated. Let $g_1,\ldots, g_r$ be generators of $D_{i-1}(A)$ which we may assume to be homogeneous, and let $K_i=\max \{\deg g_1,\ldots,\deg g_r\}$.
Consider a homogeneous element  $g\in D_{i-1}(A)$ of degree $N_i=K_i+2n-1$, then modulo the ideal $D_i(A)$ 
  $g$ is a linear combination of  elements of the type $u_ig_iv_i$, where  $u_i,v_i$ are monomials. Moreover, due to the identity
\bes
m(a,b,c)=(ma,b,c)+(m,a,bc)-(m,ab,c)-(m,a,b)c,
\ees
and its right analogue, we see that $D_{i-1}(A)D(A)+D(A)D_{i-1}(A)\subseteq D_i(A)$. Therefore, modulo $D_i(A)$, the element $u_ig_iv_i$ can be written as
\bes
u_ig_iv_i = y_1\ldots y_s g_i z_1\ldots z_t,\hbox{ where } y_j, z_k\in X,
\ees
and the arrangement of parentheses in this product does not matter. Notice that $s+t \geq 2n-1$,  hence $u_ig_iv_i\in V^{(n)}A+AV^{(n)}+D_i(A)$.
This proves the theorem.

\ctd

In the case of associative algebras, the condition of the group to be scalar is necessary for the finite generation of the subalgebra of invariants.
For alternative algebras we can prove this only in the case of zero characteristic.

\begin{cor}\label{cor2}
Let $A=Alt[V]$ be a free alternative algebra over a field of zero characteristic on a finite dimensional vector space of generators  $V$. Let  $G\subseteq GL(V)$ be a finite group whose action on $V$ is extended to $A$. Then the invariant subalgebra  $A^G$ is finitely generated if and only if  $G$ is generated by a scalar matrix.
\end{cor}
\begin{proof}
By the theorem, the sufficience of the given condition  is true in any characteristic.

Now let the subalgebra of invariants  $A^G$ be finitely  generated with the generators $f_1,\ldots,f_m$. Consider the free associative algebra $\bar A=Ass[V]$ over $V$, which is a homomorphic image of  $A$. Let $\bar f\in \bar A^G$ and $f$ be a certain pre-image of this element in $A$. Then $f-f^g\in D(A)$ for any $g\in G$, hence $nf-\sum_{g\in G}f^g\in D(A)$, where $n=|G|$. The element $tr(f)=\sum_{g\in G}f^g\in A^G=alg\la f_1,\ldots,f_n\ra$. Therefore, $\bar f\in alg\la \bar f_1,\ldots,\bar f_m\ra$, and the subalgebra of invariants $\bar A^G$ is finitely generated. Now due to \cite{Khar}, the group $G$ is generated by a scalar matrix.
\end{proof}
\bigskip

Using similar arguments and the results from  \cite{Pch1975, Hen1974}, one can prove  analogues of theorem 1  and corollary \ref{cor2} for the variety of  (-1,1)-algebras.
\smallskip

For the variety  $\rm RAlt$ of right alternative algebras the similar result is not true:
by the Dorofeev's example \cite{Dor1970} the Veronese 2-subagebra $V^{(2)}_n({{\rm RAlt}})$  in the free algebra of rank $n$
is not finitely generated.

\smallskip
We do not know whether the Veronese subalgebra  $V_k^{(n)}(\rm Jord)$ is finitely generated for a free Jordan algebra of rank  $k>2$. We can only prove that the subalgebra $V_2^{(n)}(\rm Jord)$ is finitely generated for any $n$.

\section{Notes on the freeness of nonassociative subalgebras of invariants}
\hspace{\parindent}

In this section we consider the question on the freeness of the subalgebras of invariants in varieties of nonassociative algebras.

Let $\M$ be a variety of algebras, and let   $\M_n$ denotes the subvariety of   $\M$ generated by the free   $\M$-algebra of rank $n$. Then we have the inclusions
\bes
\M_1\subseteq\M_2\subseteq\cdots\subseteq \M_n\subseteq\cdots \subseteq \M=\cup_n\M_n.
\ees
The smallest number  $n$ for which $\M=\M_n$ (if it exists) is called {\em the basic rank} of the variety $\M$. If $\M_n\neq \M$ for any natural $n$ then they say that  $\M$ has an infinite basic rank.

\begin{prop}\label{prop1}
Let $\M$ be a homogeneous not  3-nilpotent variety of algebras for which $\M_n\neq\M_{n+1}$, where $n\geq 1$ ($n\geq 2$  in anti\-com\-mu\-ta\-ti\-ve case). Then the Veronese 2-subalgebra  $V^{(2)}_n(\M)$ in the free  $\M$-algebra of rank $n$ is not free in $\M$.
\end{prop}
\begin{proof}
Consider the free $\M$-algebra $A_n=F_{\M}[x_1,\ldots,x_n]$. Let  $X_i$ denotes the set of monomials of degree $i$ on $x_1,\ldots,x_n$. It is clear that the elements of the set   $X_2\cup X_3$  should be included in any set of generators of the Veronese   2-subalgebra  $V^{(2)}(A_n)$. Since the number of elements in these sets is more then  $n$ and the algebra $A_n$ can not contain a free subalgebra of rank more than $n$, the subalgebra  $V^{(2)}(A_n)$ is not free.
\end{proof}

In particular, the Veronese  2-subalgebras are not free in the varieties of alternative, Jordan, Malcev, and  (-1,1)-algebras (see \cite{ZSSS, Sh2, Pch1975}).

\smallskip
Recall that a variety  $\M$ is called {\em a Nielsen-Schreier variety} \cite{DocUm}, if any subalgebra of a free algebra is free.

\begin{cor}\label{cor3}
Let  $\M$ be a Nielsen-Schreier variety. Then if   $\M$ is not anticommutative then its basic rank $r_b(\M)= 1$,  and in the anticommutative case  $r_b(\M)= 2$.
\end{cor}

In connection with proposition \ref{prop1} the following question arises: can the Veronese subalgebra  $V^{(k)}(\M_n)$  be free in the variety $\M_n$?  Below we will show that it is not true.

\begin{prop}\label{prop2}
 The Veronese 2-subalgebra $V^{(2)}(Alt[x,y,z])$ in the  3-generated free alternative algebra $Alt[x,y,z]$ is not free in the variety $Alt_3$.
\end{prop}
\begin{proof}
Due to results of  \cite{Sh1,Ilt}, the algebra $Alt[x,y,z]$ is isomorphic to the subalgebra of the direct sum $Ass[x,y,z]\oplus \O[X,Y,Z]$ of the free associative algebra and the free Cayley-Dickson algebra generated by the elements  $x+X,y+Y,z+Z$. It is clear that the Veronese  2-subalgebra  $V=V^{(2)}(Alt[x,y,z])$ consits of the elements $v(x,y,z)+v(X,Y,Z)$, where $v$ is an arbitrary linear combination of monomials of even degree. By theorem 1 the algebra $V$ is finitely generated. Let $f_1(x,y,z)+f_1(X,Y,Z),\ldots,f_m(x,y,z)+f_m(X,Y,Z)$ be generators of the algebra $V$. Let us show that if this system of generators is free for $V$, then its components  $f_1(x,y,z),\ldots,f_m(x,y,z)$ and  $f_1(X,Y,Z),\ldots,f_m(X,Y,Z)$ are free generators for the Veronese 2-subalgebras $V^{(2)}(Ass[x,y,z])$ and $V^{(2)}(\O[X,Y,Z])$.

Assume that the system  $f_1(x,y,z),\ldots,f_m(x,y,z)$ is not free in $Ass[x,y,z]$, then there exists an element $f(x_1,\ldots,x_m)\in Alt[x_1,\ldots,x_m]$ not lying in $D(Alt[x_1,\ldots,x_m])$ for which $f(f_1(x,y,z),\ldots,f_m(x,y,z))=0$. Let $g(x_1,\ldots,x_m)\in Alt[x_1,\ldots,x_m],\ g\in T(\O)\setminus D(Alt[x_1,\ldots,x_m])$. Then
\bes
&f(f_1(x,y,z),\ldots,f_m(x,y,z))\cdot g(f_1(x,y,z),\ldots,f_m(x,y,z))=0,&\\
&f(f_1(X,Y,Z),\ldots,f_m(X,Y,Z))\cdot g(f_1(X,Y,Z),\ldots,f_m(X,Y,Z))=0,&
\ees
which imply
\bes
f(f_1(x+X,y+Y,z+Z),\ldots,f_m(x+X,y+Y,z+Z))\cdot \\
g(f_1(x+X,y+Y,z+Z),\ldots,f_m(x+X,y+Y,z+Z))=0,
\ees
that is, $fg=0$ is an identity in the algebra $Alt[x,y,z]$. But then $fg=0$ is an identity in the algebra $Ass[x,y,z]$, a contradiction.

Similarly, using the absence of zero divisors in the free Cayley-Dickson algebra $\O[X,Y,Z]$, one can show that the system  $f_1(X,Y,Z),\ldots,$ $f_m(X,Y,Z)$ is free in $V^{(2)}(\O[X,Y,Z])$.

Therefore, if the subalgebra $V$ is free in  $Alt_3$ then the Veronese subalgebras $V^{(2)}(Ass[x,y,z])$ and $V^{(2)}(\O[X,Y,Z])$ are free as well in the varieties  $Ass$ and $Var(\O)$, respectively.

To finish the proof, it suffices now to prove that the subalgebra  $V^{(2)}(\O[X,Y,Z])$ is not free. We will show this following to \cite[Theorem 3]{Khar}.

It was proved in \cite{PolSh} that the center $Z_m$ of the free Cayley-Dickson algebra  $\O[x_1,\ldots,x_m]$ on $m$ generators has the transcendental degree  ${\rm tr.}\deg (Z_m)=7(m-2)+m$ over the basic field. It is easy to show  the inclusion
\bes
Z(V^{(2)}(\O[X,Y,Z]))\subseteq Z(\O[X,Y,Z]).
\ees
In fact, let $Z_3=Z(\O[X,Y,Z]),\ Z_3^{(2)}=Z(V^{(2)}(\O[X,Y,Z]))$. If $z\in Z_3^{(2)}$, then in the Cayley-Dickson algebra $Z_3^{-1}\O[X,Y,Z]$ over the field $Z_3^{-1}Z_3$ we have $0=[z,X^2]=tr(X)[z,X]$, which implies $[z,X]=0$.
Similarly we get that the element $z$ commutes and associates with the generators $X,Y,Z$, hence $z\in Z_3$.
Notice now that the algebra $Z_3$ is quadratic over $Z_3^{(2)}$. In fact, let $z\in Z_3^{(2)},\ z=z_0+z_1$, where $z_0$ and $z_1$ are linear combinations of monomials of even and odd length on the generators $X,Y,Z$, respectively. The center $Z_3$ is a homogeneous subalgebra, hence $z_0, z_1\in Z_3$. Then $z_0\in Z_3^{(2)}$, and we have
\bes
z^2=2z_0(z_0+z_1)+(z_1^2-z_0^2)=2z_0z+(z_1^2-z_0^2).
\ees
Since $z_1^2-z_0^2\in Z_3^{(2)}$, the element $z$ is quadratic over $Z_3^{(2)}$. Thus ${\rm tr.}\deg (Z(V^{(2)}(\O[X,Y,Z])))={\rm tr.}\deg ( Z(\O[X,Y,Z]))=10$. Therefore, if the algebra $V^{(2)}(\O[X,Y,Z])$ were free, the number of its free generators should be equal to  3. At the same time, this number should be not less then  $\dim V^{(2)}(\O[X,Y,Z])/(V^{(2)}(\O[X,Y,Z]))^2$ which is evidently bigger than 3. A contradiction.

\end{proof}

\end{document}